\documentclass[10pt, reqno]{amsart}

\usepackage[colorlinks=true]{hyperref}
\usepackage{cite, xcolor}
\usepackage{amsmath, graphicx}
\usepackage{mathtools}
\usepackage{thmtools}
\usepackage{microtype}
\usepackage{amsthm}
\usepackage{amssymb}
\usepackage{amsfonts}
\usepackage{latexsym}
\usepackage{setspace}
\usepackage{enumitem}
\usepackage{booktabs}
\usepackage{tabularx}
\usepackage{float}
\usepackage{seqsplit}
\usepackage{cancel}
\setlist[itemize]{leftmargin=*}
\setlist[enumerate]{leftmargin=*}
\setlist[enumerate, 1]{label=(\alph*), ref=\alph*}
\setlist[enumerate, 2]{label=(\roman*)}
\usepackage[font=footnotesize,labelfont=bf]{caption}

\renewcommand{\epsilon}{\varepsilon}
\renewcommand{\Pi}{\mathbb{P}}

\newtheorem{theorem}{Theorem}

\newtheorem{lemma}[theorem]{Lemma}
\theoremstyle{definition}

\numberwithin{equation}{section}

\renewcommand{\phi}{\varphi}

\newcommand{\p}{\mathfrak{p}}

\newcommand{\Z}{\mathbb{Z}}
\renewcommand{\pmod}[1]{\;(\operatorname{mod} #1)}

\renewcommand{\geq}{\geqslant}
\renewcommand{\leq}{\leqslant}

\let\oldenumerate=\enumerate
\def\enumerate{
	\oldenumerate
	\setlength{\itemsep}{5pt}
}
\let\olditemize=\itemize
\def\itemize{
	\olditemize
	\setlength{\itemsep}{5pt}
}

\allowdisplaybreaks

\title[Primitive root bias for twin primes II]{Primitive root bias for
  twin primes II: Schinzel-type theorems for totient quotients and the sum-of-divisors function}

\author{Stephan Ramon Garcia}
\address{Department of Mathematics, Pomona College, 610 N. College Ave., Claremont, CA 91711}
\email{stephan.garcia@pomona.edu}
\urladdr{\url{http://pages.pomona.edu/~sg064747}}

\author{Florian Luca}
\address{School of Mathematics, University of the Witwatersrand, Private Bag 3, Wits 2050, Johannesburg, South Africa\\
  Research Group in Algebraic Structures and Applications, King Abdulaziz University, Jeddah, Saudi Arabia\\
  Department of Mathematics, Faculty of Sciences, University of Ostrava, 30 dubna 22, 701 03
  Ostrava 1, Czech Republic}
\email{Florian.Luca@wits.ac.za}

\author{Kye Shi} \address{Department of Mathematics, Harvey Mudd
  College, 340 East Foothill Boulevard, Claremont, CA 91711-5901}
\email{kwshi@hmc.edu}

\author{Gabe Udell}
\address{Department of Mathematics, Pomona College, 610 N. College Ave., Claremont, CA 91711}
\email{grua2017@mymail.pomona.edu}

\thanks{SRG supported by NSF grant DMS-1800123.}



\begin{document}

\begin{abstract}
Garcia, Kahoro, and Luca showed that the Bateman--Horn conjecture implies $\phi(p-1) \geq \phi(p+1)$ for a majority of twin-primes pairs $p,p+2$ and that the reverse inequality holds for a small positive proportion of the twin primes.  That is, $p$ tends to have more primitive roots than does $p+2$. We prove that Dickson's conjecture, which is much weaker than Bateman--Horn, implies that the quotients $\frac{\phi(p+1)}{\phi(p-1)}$, as $p,p+2$ range over the twin primes, are dense in the positive reals.  We also establish several Schinzel-type theorems, some of them unconditional, about the behavior of $\frac{\phi(p+1)}{\phi(p)}$ and $\frac{\sigma(p+1)}{\sigma(p)}$, in which $\sigma$ denotes the sum-of-divisors function.
\end{abstract}

\maketitle

\section{Introduction}

The number of primitive roots modulo a prime $p$ is $\phi(p-1)$, in
which
\begin{equation}\label{eq:Phi}
  \phi(n) =
  \Big\lvert\big\{ i \in \{1,2,\ldots,n\} : (i,n) = 1 \big\}\Big\rvert \,=\,
  n \prod_{q \mid n} \left(1 - \frac{1}{q} \right)
\end{equation}
is the Euler totient function.  In other words, $\phi(p-1)$ is the
number of generators of the multiplicative group $(\Z/p\Z)^{\times}$.
We reserve $p,q$ for prime numbers and use $(a,b)$ to
denote the greatest common divisor of $a$ and $b$.

For twin primes $p,p+2$, it is natural to ask about the relationship
between $\phi(p-1)$ and $\phi(p+1)$.  Assuming the Bateman--Horn
conjecture, Garcia, Kahoro, and Luca proved that
\begin{equation}\label{eq:Dominant}
  \phi(p-1) \geq \phi(p+1)
\end{equation}
for a majority of the twin primes \cite{PRBTP}.  Such proportions are
computed relative to the conjectured twin-prime counting function
\begin{equation*}
  \pi _{2}(x)\,\sim\, 2C_{2}\int _{2}^{x}{dt \over (\log  t)^{2}} ,
\end{equation*}
in which
\begin{equation*}
  C_2 = \prod_{p \geq 3} \frac{p(p-2)}{(p-1)^2} \approx  0.660161815
\end{equation*}
is the \emph{twin primes constant} \cite{HL, 1CRTA}.  Here $\sim$
stands for asymptotic equivalence: $f\sim g$ means
$\lim_{x\to\infty} \frac{f(x)}{g(x)} = 1$.  The proportion of twin
primes that satisfy \eqref{eq:Dominant} is at least 65\% (assuming the
Bateman--Horn conjecture), although computations suggest something
around 98\%.  Moreover, at least 0.46\% of the twin primes satisfy the
reverse inequality $\phi(p-1) \leq \phi(p+1)$ \cite{PRBTP}.  Analogous
results for prime pairs $p, \, p+k$ were obtained by Garcia, Luca, and
Schaaff \cite{GLS}.  Garcia and Luca showed unconditionally that the
split is $50/50$ if only $p$ is assumed to be prime \cite{GL}.

\begin{table}\footnotesize
  \begin{tabular}{cc|cc|cc|cc}
    $p$ & $\frac{\phi(p+1)}{\phi(p-1)}$ & $p$ & $\frac{\phi(p+1)}{\phi(p-1)}$ & $p$ & $\frac{\phi(p+1)}{\phi(p-1)}$ & $p$ & $\frac{\phi(p+1)}{\phi(p-1)}$ \\
    \midrule
    2381 & 1.03125 & 119771 & 1.0234 & 230861 & 1.03125 & 348461 & 1.02981 \\
    3851 & 1.06 & 126491 & 1.03986 & 232961 & 1.02648 & 354971 & 1.07174 \\
    14561 & 1.05208 & 129221 & 1.06786 & 237161 & 1.01061 & 356441 & 1.04177 \\
    17291 & 1.00309 & 134681 & 1.08247 & 241781 & 1.05652 & 357281 & 1.05826 \\
    20021 & 1.11806 & 136991 & 1.03558 & 246611 & 1.0571 & 361901 & 1.09268 \\
    20231 & 1.02941 & 142871 & 1.05983 & 251231 & 1.00926 & 362951 & 1.03542 \\
    26951 & 1.06857 & 145601 & 1.05313 & 259211 & 1.01637 & 371141 & 1.02375 \\
    34511 & 1.06845 & 150221 & 1.03489 & 270131 & 1.03752 & 399491 & 1.04795 \\
    41231 & 1.05926 & 156941 & 1.04382 & 274121 & 1.06364 & 402221 & 1.09235 \\
    47741 & 1.08 & 165551 & 1.0946 & 275591 & 1.01252 & 404321 & 1.01206 \\
    50051 & 1.12 & 166601 & 1.03296 & 278741 & 1.07537 & 406631 & 1.00558 \\
    52361 & 1.13594 & 167861 & 1.06481 & 282101 & 1.08833 & 410411 & 1.13514 \\
    55931 & 1.02446 & 173741 & 1.00101 & 282311 & 1.00772 & 413141 & 1.02876 \\
    57191 & 1.05026 & 175631 & 1.05845 & 298691 & 1.037 & 416501 & 1.03179 \\
    65171 & 1.02608 & 188861 & 1.04087 & 300581 & 1.03534 & 418601 & 1.1011 \\
    67211 & 1.01413 & 197891 & 1.02266 & 301841 & 1.04082 & 424271 & 1.16905 \\
    67271 & 1.0043 & 202931 & 1.05743 & 312551 & 1.04783 & 427421 & 1.00958 \\
    70841 & 1.11799 & 203771 & 1.01071 & 315701 & 1.09613 & 438131 & 1.03357 \\
    82811 & 1.02747 & 205031 & 1.0169 & 316031 & 1.05385 & 440441 & 1.15852 \\
    87011 & 1.07857 & 205661 & 1.05097 & 322631 & 1.07177 & 448631 & 1.10491 \\
    98561 & 1.0694 & 206081 & 1.00692 & 325781 & 1.05864 & 454721 & 1.00694 \\
    101501 & 1.00679 & 219311 & 1.05694 & 328511 & 1.05523 & 464171 & 1.00607 \\
    101531 & 1.00714 & 222041 & 1.02361 & 330821 & 1.04042 & 464381 & 1.01407 \\
    108461 & 1.00871 & 225611 & 1.0726 & 341321 & 1.02666 & 465011 & 1.06779 \\
    117041 & 1.12882 & 225941 & 1.00577 & 345731 & 1.04732 & 470471 & 1.1837 \\
  \end{tabular}
  \caption{The ratio $\frac{\phi(p+1)}{\phi(p-1)}$ for the first $100$
    twin-prime pairs with $p\geq5$ for which \eqref{eq:Dominant}
    fails.  The initial numerical evidence suggests that
    $\frac{\phi(p+1)}{\phi(p-1)}$ is bounded above as $p,p+2$ runs
    over the twin primes.  Dickson's conjecture (and more extensive
    computation) suggest otherwise (Theorem~\ref{Theorem:Primitive}).}
  \label{Table:First100}
\end{table}

A glance at the numerical evidence suggests that
$\frac{\phi(p+1)}{\phi(p-1)}$ is bounded as $p, \, p+2$ range over the
twin primes; see Table~\ref{Table:First100}.  Our first theorem, whose
proof is in Section~\ref{Section:Primitive}, demonstrates that this is far
from the truth.

\begin{theorem}\label{Theorem:Primitive}
  Dickson's conjecture implies that
  \begin{equation*}
    \text{$\left\{ \frac{\phi(p+1)}{\phi(p-1)} \,:\, \text{$p$, $p+2$ are prime} \right\}$
      is dense in $[0,\infty)$}.
  \end{equation*}
\end{theorem}

Before proceeding, we require a few words about Dickson's conjecture.
The assertion that there are infinitely many twin primes is the twin
prime conjecture, which remains unresolved despite significant recent
work \cite{Maynard, Polymath, Polymath2, Zhang}.  Thus, some unproved
conjecture must be assumed to say anything nontrivial about the
large-scale behavior of the twin primes.  Dickson's conjecture is
among the weakest general assertions that implies the twin prime
conjecture \cite{Dickson, 1CRTA, Ribenboim}.

\medskip\noindent\textbf{Dickson's Conjecture.}  \emph{If
  $f_1,f_2,\ldots,f_k \in \Z[t]$ are linear polynomials with positive
  leading coefficients and $f=f_1f_2\cdots f_k$ does not vanish
  identically modulo any prime, then $f_1(t), f_2(t),\ldots,f_k(t)$
  are simultaneously prime infinitely often.}  \medskip

The twin prime conjecture is the special case $f_1(t) = t$ and
$f_2(t) = t+2$.  Dickson's conjecture is weaker than the Bateman--Horn
conjecture, which concerns polynomials of arbitrary degree and makes
asymptotic predictions \cite{Bateman,Bateman2,1CRTA}.  More extensive
computations suggest the truth of Theorem~\ref{Theorem:Primitive}; see
Table~\ref{Table:Leaders}.

\begin{table}\footnotesize
  \begin{tabular}{cc|cc}
    $p$ & $\frac{\phi(p+1)}{\phi(p-1)}$ & $p$ & $\frac{\phi(p+1)}{\phi(p-1)}$ \\[3pt]
    \midrule
    2381 & 1.03125 & 17497481 & 1.27427 \\
    3851 & 1.06 & 69989921 & 1.27484 \\
    20021 & 1.11806 & 78278201 & 1.28693 \\
    50051 & 1.12 & 183953771 & 1.30984 \\
    52361 & 1.13594 & 242662421 & 1.32797 \\
    424271 & 1.16905 & 468818351 & 1.34577 \\
    470471 & 1.1837 & 2156564411 & 1.37262 \\
    602141 & 1.18793 & 24912037151 &1.37901\\
    2302301 & 1.2058 & 43874931101 &1.37949\\
    6806801 & 1.23097 & 73769375681 &1.39837 \\
    16926911 & 1.23678 & 131104243271 &1.42545
  \end{tabular}
  \caption{Running leaders among twin primes $p,p+2$ with $p\geq 5$
    for which \eqref{eq:Dominant} fails.  Theorem~\ref{Theorem:Primitive}
    suggests that $\frac{\phi(p+1)}{\phi(p-1)}$ can be arbitrarily
    large.}
  \label{Table:Leaders}
\end{table}

Totient quotients have a long and storied history \cite[Ch.~1]{Sandor}.
Schinzel established a curious result in 1954 \cite{Schinzel}, when he
showed that
\begin{equation}\label{eq:Schinzel}
  \text{$\left\{ \frac{\phi(n+1)}{\phi(n)} :  n =1,2,\ldots \right\}$ is dense in $[0,\infty)$}.
\end{equation}
This inspired later research by Schinzel, Sierpi\'{n}ski, Erd\H{o}s,
and others \cite{Schinzel2,Schinzel3, SS,SW,Erdos,Erdos2}.

The prime analogue of \eqref{eq:Schinzel} is false since
$\limsup_{p\to\infty} \frac{\phi(p+1)}{\phi(p)} \leq \frac{1}{2}$
because $p+1$ is even when $p$ is odd.  Taking this into account, we
prove in Section~\ref{Section:SchinzelPrime} that the following modified
analogue of Schinzel's theorem holds unconditionally.  The main
ingredient is a generalization of Chen's theorem
\cite[Thm.~25.11]{Friedlander}.

\begin{theorem}\label{Theorem:SchinzelPrime}
  (Unconditional)
  \[\text{$\displaystyle \left\{ \frac{\phi(p+1)}{\phi(p)} \,:\, \text{$p$
      prime} \right\}$ is dense in $\left[0,\tfrac{1}{2}\right]$}.\]
\end{theorem}

The corresponding twin-prime analogue of Schinzel's theorem
\eqref{eq:Schinzel} is the following result, whose proof is in
Section~\ref{Section:Schinzel3}.

\begin{theorem}\label{Theorem:Schinzel3}
  Dickson's conjecture implies that
  \[
    \text{$\left\{ \frac{\phi(p+1)}{\phi(p)} \,:\, \text{$p$, $p+2$
          prime} \right\}$ is dense in $\left[0,\tfrac{1}{3}\right]$}.
  \]
\end{theorem}

Our proofs are transparent enough to permit the construction of
striking numerical examples that cannot be obtained easily through
brute force alone.  For example, the twin primes
\begin{quote}\footnotesize
  $ p=  \seqsplit{7642856398602124688629749934198565871312540429046303895770916192951906734870659150561934966844646027084062815031558255187497845592422263591099165956612523533130293687015551343007872097253311773591110917528188316414599672498791126998631571669927514554696197257278634275128724279128189209746271781127728971190835766821004705695431931600462680599536653440211216644627374103340174280330773185320397138921513211749572729188424740776331734373496916520151521002589283157543831117944688070785572501766739882905425624193096685596320850078698276901411040453944941282023066485167044354346570749231118823789564541354721325254528236757389662894019979784752005836753790957246918204110303873795632580172005942425384506696512023749191214954151888293569682377697953981506694554048884224168645065101150446101000833424954728109752271343988672243937380333654916021273646288933510923983231864735583985112984704333900355465514181712760200823033200490940934697445099785257641185993998522771636033743449797832811518401162205004091941408725193787590452993251183485650458246694103753512762453446893673884455566405601755508021}
  $
\end{quote}
and $p+2$ yield a ratio $\frac{\phi(p+1)}{\phi(p-1)} = 3.11615\ldots$,
which is far larger than those displayed in Table~\ref{Table:Leaders}.
As another example, consider
$\frac{\pi}{10} = 0.31415\ldots \in \left[0,\frac{1}{3}\right]$.  The
method of proof of Theorem~\ref{Theorem:Schinzel3} (with slight
modifications) and a computer search yields the the twin prime
pair $p = 7726274821004474852086566160138278575763701613133157$ and
$p+2$, which satisfies (the underlined digits agree with those of
$\frac{\pi}{10}$)
\begin{equation*}
  \frac{\phi(p+1)}{\phi(p)} = \underline{0.314159265358979}21341\ldots.
\end{equation*}

Theorem~\ref{Theorem:Primitive}, Theorem~\ref{Theorem:SchinzelPrime}, and
Theorem~\ref{Theorem:Schinzel3} each have analogues for the sum-of-divisors
function $\sigma(n) = \sum_{d \mid n} d$.  We collect these results in
the following theorem, whose proof is in Section~\ref{Section:Sigma}.

\begin{theorem}
  \label{Theorem:Sigma}
  \mbox{}
  \begin{enumerate}
  \item \label{Part:SigmaPrimitive}
    Dickson's conjecture implies that
    \[
      \text{$\left\{ \frac{\sigma(p+1)}{\sigma(p-1)} \,:\,
          \text{$p$, $p+2$ prime} \right\}$ is dense in
        $[0,\infty)$}.
    \]
  \item \label{Part:SigmaSchinzelPrime}
    (Unconditional)
    \[
      \text{$\displaystyle\left\{ \frac{\sigma(p+1)}{\sigma(p)} :
          \text{$p$ prime} \right\}$ is dense in
        $\left[\tfrac{3}{2},\infty\right)$}.
    \]
  \item \label{Part:SigmaSchinzel3}
    Dickson's conjecture implies that
    \[
      \text{$\left\{ \frac{\sigma(p+1)}{\sigma(p)} \,:\,
          \text{$p$, $p+2$ prime} \right\}$ is dense in
        $[2,\infty)$}.
    \]
  \end{enumerate}
\end{theorem}

\section{Proof of
  Theorem~\ref{Theorem:Primitive}}\label{Section:Primitive}

\subsection*{A folk lemma}
Mertens' third theorem asserts that
\begin{equation}\label{eq:Mertens}
  \prod_{q\leq x}\left(1- \frac{1}{q}\right) \,\sim\, \frac{e^{-\gamma}}{\log x},
\end{equation}
in which $\gamma$ is the Euler--Mascheroni constant \cite{Mertens,
  Hardy}.  A more elementary proof of the following lemma can be based
on \cite[Prop.~8.8]{DeKoninck} instead.

\begin{lemma}\label{Lemma:Phi}
  Let $\mathcal{P}$ denote a finite set of primes.  Then
  \begin{equation*}
    \left\{ \frac{\phi(n)}{n} : \text{$n$ squarefree, $p \nmid n$ for
        all $p \in \mathcal{P}$}\right\}
  \text{ is dense in $[0,1]$}.
  \end{equation*}
\end{lemma}

\begin{proof}
    Let $\xi \in (0,1]$ and
  $n_t = \prod_{e^{\xi t} \leq q < e^t } q$, in which
  $t > \frac{1}{\xi} \log \max \mathcal{P}$.  Then $n_t$ is
  squarefree, $p \nmid n_t$ for all $p \in \mathcal{P}$, and
  \begin{equation*}
    \frac{\phi(n_t)}{n_t} 
    = \prod_{e^{\xi t} \leq q < e^t}\left(1 - \frac{1}{q}\right)
    \sim \frac{e^{-\gamma} / \log (e^t)}{e^{-\gamma} / \log (e^{\xi t})}
    \sim \frac{\log (e^{\xi t})}{\log (e^t)} = \xi
  \end{equation*}
  as $t \to \infty$.
\end{proof}

\subsection*{Initial setup}
It suffices to show that Dickson's conjecture implies that for each
fixed $\xi \in (0,\infty)$ and $0<\delta < 1$, there is a twin-prime
pair $p,p+2$ such that
$\frac{\phi(p+1)}{\phi(p-1)} \in
\big(\xi(1-\delta),\xi(1+\delta)\big)$.

Let $0 < x \leq \min\left\{\frac{2}{3},\xi\right\}$.
Lemma~\ref{Lemma:Phi} provides a squarefree $b$ such that
\begin{equation*}
  (b,6) = 1 \quad \text{and} \quad
  \frac{\phi(b)}{b}\in \left(\frac{x}{\xi}\left(\frac{1}{1+\delta}\right),\frac{x}{\xi} \right).
\end{equation*}
A second appeal to Lemma~\ref{Lemma:Phi} yields a squarefree $a'$ such
that
\begin{equation*}
  (a',6b) = 1 \quad \text{and}\quad
  \frac{\phi(a')}{a'}  \in \left(\frac{3x}{2} \, (1- \delta),\frac{3x}{2}\right).
\end{equation*}
Our choice of $x$ ensures that the intervals specified are contained
in $(0,1)$.  Let $a = 3a'$ and observe that
\begin{equation*}
  \frac{\phi(a)}{a} = \frac{2}{3} \frac{\phi(a')}{a'} \in
  \big(x(1-\delta),x\big).
\end{equation*}
Consequently,
\begin{equation}\label{eq:Trick}
  \frac{\phi(a)}{a} \frac{b}{\phi(b)}
  \,\in\,  \big(\xi(1-\delta),\xi(1+\delta) \big).
\end{equation}

\subsection*{The polynomials}\label{Section:Polynomials}
Our strategy is to produce linear polynomials $f_1, f_2, \dots, f_4$
to which Dickson's conjecture can be applied, using $f_1, f_2$ to
produce twin primes $p, \, p+2$, and using $f_3, f_4$ to ensure
that $\frac{\phi(p+1)}{\phi(p-1)}$ falls in the desired interval.

Since $8,a^2,b^2$ are pairwise relatively prime, the Chinese remainder
theorem provides $c$ such that
\begin{align}
  c  &\equiv  5\pmod 8, \label{eq:pc58}\\
  c  &\equiv  a-1\pmod {a^2}, \quad\text{and} \label{eq:pca12}\\
  c  &\equiv  b+1\pmod {b^2}. \label{eq:pcb1}
\end{align}
Since $3 \mid a$, it follows from \eqref{eq:pca12} that
\begin{equation}\label{eq:pc23}
  c  \equiv  2 \pmod 3.
\end{equation}
Define
\begin{equation*}
  h(t)=24a^2b^2t + c
\end{equation*}
and let
\begin{align*}
  f_1(t) & =  h(t),\\
  f_2(t) & =  h(t)+2,\\
  f_3(t) &= \frac{h(t)-1}{4b} = \frac{24a^2b^2t + (c-1)}{4b} = 6a^2b t + \frac{c-1}{4b},\quad \text{and}\\
  f_4(t) &= \frac{h(t)+1}{2a} =  \frac{24a^2b^2t + (c+1)}{2a}  = 12ab^2 t + \frac{c+1}{2a},
\end{align*}
Clearly $f_1,f_2 \in \Z[t]$.  Observe that \eqref{eq:pc58} and
\eqref{eq:pcb1} ensure that $f_3$ has integral coefficients.
Similarly, \eqref{eq:pc58} and \eqref{eq:pca12} ensure that $f_4$ has
integral coefficients.  Thus, all four polynomials are in $\Z[t]$ and
have positive leading coefficients.

\subsection*{Nonvanishing product}
We claim that $f = f_1f_2f_3f_4$ does not vanish identically modulo
any prime.  Since
\begin{align*}
  f_1(t) &\equiv f_2(t) \equiv c \equiv 1 \pmod{2}, && \text{by \eqref{eq:pc58}},\\
  f_3(t) &\equiv \frac{c-1}{4b} \equiv 1 \pmod{2}, && \text{by $(b,2) = 1$ and \eqref{eq:pc58}},\\
  f_4(t) &\equiv \frac{c+1}{2a} \equiv  1 \pmod{2}, &&\text{by $(a,2) = 1$ and \eqref{eq:pc58}},
\end{align*}
it follows that $f$ does not vanish modulo $2$.  Similarly,
\begin{align*}
  f_1(t) &\equiv c \equiv 2 \pmod{3}, && \text{by \eqref{eq:pc23}},\\
  f_2(t) &\equiv c+2 \equiv 1 \pmod{3}, && \text{by \eqref{eq:pc23}},\\
  f_3(t) &\equiv \frac{c-1}{4b} \equiv b \not\equiv 0 \pmod{3}, && \text{by $(b,3)=1$ and \eqref{eq:pc23}},\\
  f_4(t) &\equiv \frac{c+1}{2a} \equiv 2 \pmod{3}, &&\text{by \eqref{eq:pca12}},
\end{align*}
so $f$ does not vanish modulo $3$.  The final statement perhaps
deserves a bit of explanation.  From \eqref{eq:pca12} we have
$c+1 \equiv a \pmod{a^2}$ and hence $\frac{c+1}{a} \equiv 1 \pmod{a}$.
Since $3 \mid a$, it follows that $\frac{c+1}{a} \equiv 1 \pmod{3}$
from which the desired statement follows.

For any prime $q\geq 5$ such that $q\nmid ab$, the polynomial $f$ has
degree four and hence cannot vanish identically modulo $q$.  Now
suppose that $q\geq 5$ is prime and $q \mid ab$.  Then
$h(t) \equiv c \pmod{q}$.  Since \eqref{eq:pca12} and \eqref{eq:pcb1}
ensure that
\begin{equation}\label{eq:pc11}
  c \equiv
  \begin{cases}
    -1 \pmod{q}& \text{if $q \mid a$},\\
    1 \pmod{q}& \text{if $q \mid b$},
  \end{cases}
\end{equation}
it follows that $f_1$ and $f_2$ do not vanish modulo $q$.  Similarly,
\begin{align*}
  f_3(t) &\equiv \frac{c-1}{4b} \equiv
           \begin{cases}
             -2^{-1}b^{-1} \pmod{q} & \text{if $q \mid a$ (by \eqref{eq:pc11})},\\
             4^{-1} \pmod{q} & \text{if $q \mid b$ (by
               \eqref{eq:pcb1})},
           \end{cases}
  \\
  f_4(t) &\equiv \frac{c+1}{2a} \equiv
           \begin{cases}
             2^{-1} \pmod{q} & \text{if $q \mid a$ (by \eqref{eq:pca12})},\\
             a^{-1} \pmod{q} & \text{if $q \mid b$ (by
               \eqref{eq:pc11})}.
           \end{cases}
\end{align*}
Thus, $f$ does not vanish identically modulo any prime.

\subsection*{Conclusion}
Dickson's conjecture provides infinitely many $T$ such that $f_1(T)$,
$f_2(T)$, $f_3(T)$, and $f_4(T)$ are prime.  For such $T$, the primes
\begin{equation*}
  p=f_1(T)\qquad \text{and}\qquad p+2=f_2(T)
\end{equation*}
satisfy
\begin{equation*}
  p+1 =  2a \, f_4(T)
  \qquad\text{and}\qquad
  p-1 = 4b \, f_3(T).
\end{equation*}
Consequently, \eqref{eq:Trick} ensures that
\begin{align*}
  \frac{\phi(p+1)}{\phi(p-1)}
  &=\frac{\phi(p+1)}{p+1}  \frac{p-1}{\phi(p-1)}  \frac{p+1}{p-1} \\
  & =  \frac{\phi(p+1)}{p+1}  \frac{p-1}{\phi(p-1)} \big(1+o(1)\big)\\
  & =   \frac{\phi(2a \, f_4(T))}{2a \, f_4(T)}
    \frac{4b \, f_3(T)}{\phi(4b \, f_3(T))} \big(1+o(1)\big)\\
  & =   \frac{\phi(2) \, \phi(a) \, \phi(f_4(T))}{2a \, f_4(T)}  \frac{4b \, f_3(T)}{\phi(4) \, \phi(b) \, \phi(f_3(T))} \big(1+o(1)\big)\\
  & =  \frac{1 \phi(a)}{2a}  \frac{f_4(T)-1}{f_4(T)}  \frac{4b}{2\phi(b)}  \frac{f_3(T)}{f_3(T)-1}  \big(1+o(1)\big)\\
  & =   \frac{\phi(a)}{a}\frac{b}{\phi(b)}  \big(1+o(1)\big)
\end{align*}
belongs to $\big(\xi(1-\delta),\xi(1+\delta)\big)$ for large $T$.
Here we have used the facts
$(2a,f_4(T)) = 1$ and $(4b,f_3(T)) = 1$, which  
follow from \eqref{eq:pc58} and \eqref{eq:pca12}, and from
\eqref{eq:pc58} and \eqref{eq:pcb1}, respectively.
\qed

\section{Proof of
  Theorem~\ref{Theorem:SchinzelPrime}}\label{Section:SchinzelPrime}

Chen's theorem asserts that every sufficiently large even number is
the sum of two primes, or a sum of a prime and a semiprime (a number
with precisely two prime factors) \cite{Chen1, Chen2}.  We require a
generalization of Chen's theorem to linear forms.  The version below
is due to Friedlander and Iwaniec \cite[Thm.~25.11]{Friedlander}.

\begin{theorem}[Chen, Friedlander--Iwaniec]\label{Theorem:Chen}
  Let $a,c \geq 1$ and $b\neq 0$ be pairwise coprime integers with
  $2 \mid abc$.  For $t$ sufficiently large (in terms of $abc$),
  \begin{equation*}
    \big\lvert  \{ p \leq t : ap+b = cs\} \big \rvert \, \geq \, \frac{W(abc)}{31c}  \frac{Bt}{(\log t)^2},
  \end{equation*}
  in which $s$ has at most two prime factors, each one larger than
  $t^{3/11}$, 
  \begin{equation*}
    W(d) = \prod_{\substack{p \mid d\\p>2}} \left(1 - \frac{1}{p-1}\right)^{-1}
    \qquad\text{and}\qquad
    B = 2\prod_{p\geq 3} \left(1 - \frac{1}{(p-1)^2} \right).
  \end{equation*}
\end{theorem}

Let $\xi \in \left[0,\frac{1}{2}\right]$ and $\delta > 0$.  For
$x\geq \frac{\log 2}{2\xi}$, the integer $Q = Q(x)$ defined by
\begin{equation*}
  Q = 2 \prod_{\mathclap{e^{2\xi x} < q \leq e^x}} q
\end{equation*}
is divisible by $2$, but not $4$.  As $x \to \infty$, \eqref{eq:Phi}
and \eqref{eq:Mertens} imply
\[
  \frac{\phi(Q)}{Q} = \frac{1}{2} \prod_{e^{2\xi x} < q \leq e^x}
  \left(1 - \frac{1}{q} \right) \sim \frac{1}{2}
  \frac{e^{-\gamma}}{\log(e^x)} \frac{\log(e^{2\xi x})}{e^{-\gamma}} =
  \xi.
\]
For each $x$, apply Theorem~\ref{Theorem:Chen} with $a = b = 1$
and $c = Q$ to obtain an $S = S(x)$ with at most two prime factors,
both of which are greater than $\max\{Q,x\}$, such that
$p = p(x) = QS-1$ is prime.  Then
\begin{equation*}
  \lim_{x\to\infty} \frac{\phi(S)}{S} = \lim_{x\to\infty}  \prod_{q \mid S}\left(1 - \frac{1}{q}\right) = 1
\end{equation*}
and hence
\begin{align*}
  \frac{\phi(p+1)}{\phi(p)}
  = \frac{ \phi(QS) }{\phi(QS-1)}
  = \frac{ \phi(Q) \, \phi(S) }{ QS-2}
  = \frac{\phi(Q)}{Q}  \frac{\phi(S)}{S} \frac{S}{S-2/Q} \to \xi
\end{align*}
as $x\to \infty$.  This concludes the proof. \qed

\section{Proof of Theorem~\ref{Theorem:Schinzel3}
}\label{Section:Schinzel3}

Fix $\xi \in \left(0,\frac{1}{3}\right)$ and let
$0 < \delta < \frac{1 - 3 \xi}{3 \xi}$.  Lemma~\ref{Lemma:Phi} yields
a squarefree $Q'$ such that
\begin{equation*}
  (Q',6) = 1 \qquad \text{and} \qquad
  \frac{\phi(Q')}{Q'} \in \big(3\xi(1-\delta),3\xi(1+\delta)\big).
\end{equation*}
Let $Q=6Q'$, and observe that
\[
  \frac{\phi(Q)}{Q} = \frac 1 3 \frac{\phi(Q')}{Q'}
  \in \big(\xi(1-\delta), \xi(1+\delta)\big).
\]
Define the polynomials
\begin{equation*}
  f_1(t) = t,\qquad
  f_2(t) = Qt-1,\quad \text{and} \quad
  f_3(t) = Qt+1.
\end{equation*}
If $q \geq 5$ and $q \nmid Q$, then $f= f_1f_2f_3$ has degree three
and cannot vanish identically modulo $q$.  If $q \mid Q$, then
$f(1) = Q^2 - 1 \equiv -1 \pmod{q}$ and hence $f$ does not vanish
identically modulo $q$.  In particular, $f$ does not vanish
identically modulo $2$ or $3$.  Thus, $f$ does not vanish identically
modulo any prime.

Dickson's conjecture provides infinitely many $T$ such that $f_1(T)$,
$f_2(T)$, and $f_3(T)$ are prime.  In particular, we may assume that
the prime $f_1(T) =T$ is greater than $Q$ so that $(Q,T) = 1$.  Then $p=QT-1$ and $p+2=QT+1$ are twin
primes and $p+1 = QT$.  Then
\begin{align*}
  \frac{\phi(p+1)}{\phi(p)} &= \frac{\phi(QT)}{\phi(QT-1)}\\
                            &= \frac{\phi(Q) \, \phi(T)}{\phi(QT-1)}\\
                            &= \frac{\phi(Q) (T-1)}{QT-2} \\
                            &= \frac{\phi(Q)}{Q} \frac{T-1}{T-2/Q} \\
                            &= \frac{\phi(Q)}{Q} \big(1 + o(1)\big)
\end{align*}
is in $\big(\xi(1-\delta),\xi(1+\delta)\big)$ for sufficiently large
$T$.  \qed

\section{Proof of Theorem~\ref{Theorem:Sigma}}\label{Section:Sigma}

\subsection*{Proof of Theorem~\ref{Theorem:Sigma}\ref{Part:SigmaPrimitive}}
The proof of Theorem~\ref{Theorem:Sigma}\ref{Part:SigmaPrimitive} is
similar to the proof of Theorem~\ref{Theorem:Primitive}.  We first require
the following version of Lemma~\ref{Lemma:Phi} for the sum-of-divisors
function.

\begin{lemma}\label{Lemma:Sigma}
  Let $\mathcal{P}$ denote a finite set of primes.  Then
  \begin{equation*}
    \left\{ \frac{\sigma(n)}{n} : \text{$n$ squarefree, $p \nmid n$ for all $p \in \mathcal{P}$}\right\}
  \text{ is dense in $[1,\infty)$}.
  \end{equation*}
\end{lemma}

\begin{proof}
  Let $Q = Q(t) = \prod_{q\leq t} q$.  Then Mertens' third theorem
  \eqref{eq:Mertens}, the Euler product formula, and the evaluation
  $\zeta(2) = \frac{\pi^2}{6}$ yield
  \begin{equation*}
    \frac{\sigma(Q)}{Q}
    = \prod_{q \leq t}\frac{1+q}{q}
    = \prod_{q \leq t}\left(1 + \frac{1}{q} \right)
    = \frac{\prod_{q \leq t}(1 - 1/q^2 )}{\prod_{q \leq t}(1 - 1/q)}
    \sim \frac{6/\pi^2}{e^{-\gamma}/\log t} = \frac{6e^{\gamma}}{\pi^2} \log t
  \end{equation*}
  as $t \to \infty$.  
  Let $\xi \geq 1$ and define
  $n_t = \prod_{e^t \leq q < e^{\xi t}} q$, in which
  $\log t > \max \mathcal{P}$.
  Then $n_t$ is squarefree, $p \nmid n_t$ for all $p \in \mathcal{P}$, and
  \begin{equation*}
    \frac{\sigma(n_t)}{n_t} \sim \frac{\log (e^{\xi t})}{\log (e^t)} = \xi
  \end{equation*}
  as $t \to \infty$.
\end{proof}

Fix $\xi \in (0,\infty)$ and $0<\delta < 1$.  Let
$x \geq \max\big\{\frac 4 3,\frac{7\xi}{6}\big\}$.  Then
Lemma~\ref{Lemma:Sigma} provides a squarefree $b$ such that
\begin{equation*}
  (b,6) = 1 \quad \text{and} \quad
  \frac{\sigma(b)}{b}\in \left(\frac{6x}{7\xi},\frac{6x}{7\xi}\left(\frac{1}{1-\delta}\right) \right).
\end{equation*}
A second appeal to Lemma~\ref{Lemma:Sigma} yields a squarefree $a'$
such that
\begin{equation*}
  (a',6b) = 1 \quad \text{and}\quad
  \frac{\sigma(a')}{a'}  \in \left(\frac{3x}{4},\frac{3x}{4} \, (1+ \delta)\right).
\end{equation*}
Our choice of $x$ ensures that the intervals specified are contained
in $(1,\infty)$. Let $a = 3a'$ and observe that
\begin{equation*}
  \frac{\sigma(a)}{a} = \frac{4}{3} \frac{\sigma(a')}{a'} \in
  \big(x,x(1+ \delta)\big).
\end{equation*}
Consequently,
\begin{equation}\label{eq:Trick67}
  \frac{\sigma(a)}{a} \frac{b}{\sigma(b)}
  \in  \left(\frac{7}{6}\,\xi(1-\delta),\frac{7}{6}\,\xi(1+\delta) \right).
\end{equation}

Define the polynomials $f_1,f_2,f_3,f_4$ as in the proof of
Theorem~\ref{Theorem:Primitive}, in which we showed that the application
of Dickson's conjecture to this family is permissible.  Dickson's
conjecture provides infinitely many $T$ such that $f_1(T)$, $f_2(T)$,
$f_3(T)$, and $f_4(T)$ are prime.  For such $T$, the primes
\begin{equation*}
  p=f_1(T)\qquad \text{and}\qquad p+2=f_2(T)
\end{equation*}
satisfy $p+1 = 2a \, f_4(T)$ and $p-1 = 4b \, f_3(T)$.  Consequently,
\eqref{eq:Trick67} ensures that
\begin{align*}
  \frac{\sigma(p+1)}{\sigma(p-1)}
  &=\frac{\sigma(p+1)}{p+1} \frac{p-1}{\sigma(p-1)} \frac{p+1}{p-1} \\
  &=  \frac{\sigma(p+1)}{p+1} \frac{p-1}{\sigma(p-1)} \big(1+o(1)\big)\\
  &=   \frac{\sigma(2a \, f_4(T))}{2a \, f_4(T)}
    \frac{4b \, f_3(T)}{\sigma(4b \, f_3(T))} \big(1+o(1)\big)\\
  &= \frac{\sigma(2) \, \sigma(a) \,\sigma(f_4(T))}{2af_4(T)}
    \frac{4b \, f_3(T)}{\sigma(4) \, \sigma(b) \, \sigma(f_3(T))} \big(1+o(1)\big)\\
  &=  \frac{3\sigma(a)}{2a}  \frac{f_4(T)+1}{f_4(T)}  \frac{4b}{7 \sigma(b)}  \frac{f_3(T)}{f_3(T)+1} \big(1+o(1)\big)\\
  & =  \frac{6}{7}  \frac{\sigma(a)}{a}\frac{b}{\sigma(b)}  \big(1+o(1)\big)
\end{align*}
belongs to $\big(\xi(1-\delta),\xi(1+\delta)\big)$ for large $T$. \qed

\subsection*{Proof of Theorem~\ref{Theorem:Sigma}\ref{Part:SigmaSchinzelPrime}}
Since the proof of
Theorem~\ref{Theorem:Sigma}\ref{Part:SigmaSchinzelPrime} is similar to the
proof of Theorem~\ref{Theorem:SchinzelPrime}, we only sketch the details.
First, a simple modification of Lemma~\ref{Lemma:Sigma} shows that for
any finite set $\mathcal{P}$ of primes that does not contain $2$, the
set
\[
  \left\{ \frac{\sigma(n)}{n} : \text{$n$ squarefree and even,
      $p \nmid n$ for all $p \in \mathcal{P}$}\right\}
\]
is dense in $\left[\frac{3}{2},\infty\right)$.  Let
$\xi \in \left[\frac{3}{2},\infty\right)$ and mimic the proof of
Theorem~\ref{Theorem:SchinzelPrime} to find an even squarefree $Q = Q(x)$
such that $\frac{\sigma(Q)}{Q} \to \xi$ as $x\to \infty$.  Apply
Theorem~\ref{Theorem:Chen} and obtain an $S = S(x)$ with at most two prime
factors, both of which are greater than $\max\{Q,x\}$, such that
$p = p(x) = QS-1$ is prime.  Then $\frac{\sigma(S)}{S}\to1$ as
$x \to \infty$ and hence
\begin{equation*}
  \frac{\sigma(p+1)}{\sigma(p)} = \frac{ \sigma(QS) }{\sigma(QS-1)} =
  \frac{ \sigma(Q) \, \sigma(S) }{ QS} = \frac{\sigma(Q)}{Q}
  \frac{\sigma(S)}{S} \to \xi. \hfill\qed
\end{equation*}

\subsection*{Proof of Theorem~\ref{Theorem:Sigma}\ref{Part:SigmaSchinzel3}}
Since the proof of Theorem~\ref{Theorem:Sigma}\ref{Part:SigmaSchinzel3} is
similar to the proof of Theorem~\ref{Theorem:Schinzel3}, we only sketch
the details.  Let $\xi\in [2,\infty)$ and mimic the proof of
Theorem~\ref{Theorem:Schinzel3} to find a squarefree $Q=Q(x)$ which is
divisible by 6 such that $\frac{\sigma(Q)}{Q}\to \xi$ as
$x\to \infty$.
	
Define the polynomials $f_1,f_2,f_3$ as in the proof of
Section~\ref{Section:Schinzel3} in which we showed that the application of
Dickson's conjecture to this family is permissible. Thus, we can find
arbitrarily large $T$ such that
$f_1(T)=T$, $p=f_2(T)=QT-1$, and $p+2=f_3(T)=QT+1$ are simultaneously
prime and hence
\[
  \frac{\sigma(p+1)}{\sigma(p)}=\frac{\sigma(QT)}{\sigma(QT-1)} =
  \frac{\sigma(Q) \, \sigma(T)}{QT}=\frac{\sigma(Q)}{Q} \frac{T+1}{T}
  \to \xi.\qed
\]

\section{Numerical examples}

Our methods of proof are transparent enough that they permit us to
construct numerical examples whose
totient and divisor-sum quotients approximate various mathematical
constants surprisingly well (much better than can be obtained by
brute force alone).  Tables \ref{Table:Primitive},
\ref{Table:SchinzelPrime}, \ref{Table:Schinzel3},
\ref{Table:SigmaPrimitive}, \ref{Table:SigmaSchinzelPrime},
and  \ref{Table:SigmaSchinzel3} showcase various examples for
each of the theorems proven above.

\subsection*{Computational differences}

For the sake of optimization, our computation of numerical examples involves
slightly different methods than those provided in the
proofs.
In particular, our provided proofs of Lemmas \ref{Lemma:Phi} and
\ref{Lemma:Sigma} construct a product of consecutive
primes between $e^{\xi t}$ and $e^t$.  
Our computation takes a more na\"ive but more efficient
process: begin with $1$, and repeatedly multiply by the
next smallest $q \notin \mathcal{P}$ so that
$\frac{\phi(n)}{n} \geq \xi$ (resp., for Lemma~\ref{Lemma:Sigma},
$\frac{\sigma(n)}{n} \leq \xi$); convergence of this process is
guaranteed by the fact that $\prod_q (1-1/q)$ diverges to $0$ (resp.,
$\prod_q (1+1/q)$ diverges to $\infty$), so the sequence we construct is
monotonically decreasing (resp., increasing) and is bounded tightly
below (resp., above) by $\xi$.

For Theorem~\ref{Theorem:Primitive}, Theorem~\ref{Theorem:Schinzel3},
Theorem~\ref{Theorem:Sigma}\ref{Part:SigmaPrimitive}, and
Theorem~\ref{Theorem:Sigma}\ref{Part:SigmaSchinzel3}, the method of
construction is otherwise the same, relying on the same
polynomial-based approach together with Dickson's conjecture.  For
Theorem~\ref{Theorem:SchinzelPrime} and
Theorem~\ref{Theorem:Sigma}\ref{Part:SigmaSchinzelPrime}, instead of the
unconditional method of proof based on Theorem~\ref{Theorem:Chen} provided
in the paper, we instead took a polynomial/Dickson approach similar to
that of Theorem~\ref{Theorem:Schinzel3} and
Theorem~\ref{Theorem:Sigma}\ref{Part:SigmaSchinzel3} based on
Lemma~\ref{Lemma:Phi} and Lemma~\ref{Lemma:Sigma}, since we found no
straightforward numerical implementation of Theorem~\ref{Theorem:Chen}.

\newcommand{\pwidth}{\widthof{0000000000000000000000000}}
\renewcommand{\arraystretch}{1.5}

\begin{table}[H]\footnotesize
  \[
    \begin{array}{r@{\;=\;}l|m{\pwidth-\widthof{00000}}l}
      \multicolumn{2}{c|}{\xi}
      & \multicolumn{1}{c}{p}
      & \multicolumn{1}{c}{\frac{\phi(p+1)}{\phi(p-1)}} \\ \midrule
      \gamma
      & 0.577215664901532960\dots
      & \seqsplit{423723808834634215177323136310347257917981529437}
      & \underline{0.5772156649015328}95\dots \\
      \frac \pi {10}
      & 0.31415926535897932\dots
      & \seqsplit{8219517414757353575739578784498983486285599838387837563816565056272351869}
      & \underline{0.314159265358979}29\dots \\
      \frac \pi 2
      & 1.570796326794\dots
      & \seqsplit{166699917024666766393292717283478616682999425902809932261}
      & \underline{1.5707963267}82\dots \\
      \frac e {10}
      & 0.271828182845904523\dots
      & \seqsplit{340751879321709831822472985002608295768606668541034063696514572757394102608704519864613389}
      & \underline{0.2718281828459045}01\dots \\
      \frac 1 e
      & 0.36787944117144232\dots
      & \seqsplit{1136836570875009647818990149600716232565777311442730480596395907730326829}
      & \underline{0.367879441171442}29\dots \\
      \sqrt 2
      & 1.414213562373095048\dots
      & \seqsplit{99936342969417150404929694194402869221132671592712932741}
      & \underline{1.4142135623730950}34\dots \\
      \sqrt 3
      & 1.732050807568877\dots
      & \seqsplit{6844049149066600460918767470146205475936305896538422749570651356283525048962967661}
      & \underline{1.7320508075688}62\dots \\
      \frac {\sqrt 5+1} 2
      & 1.6180339887498948\dots
      & \seqsplit{13069649048652430795143497540458872226218189143056594811839738637429981}
      & \underline{1.61803398874989}32\dots \\
      \frac {\sqrt 5-1} 2
      & 0.61803398874989484\dots
      & \seqsplit{113819885725207836594181379451145300826370298475149690990486037}
      & \underline{0.618033988749894}93\dots \\
      \log 2
      & 0.693147180\dots
      & \seqsplit{763278080829361407712397}
      & \underline{0.6931471}72\dots \\
      \log 3
      & 1.09861228866810969\dots\dots
      & \seqsplit{86435370621522915217536274921197129646793034461}
      & \underline{1.098612288668109}05\dots \\
    \end{array}
  \]
  \caption{Numerical examples for Theorem~\ref{Theorem:Primitive}.}
  \label{Table:Primitive}
\end{table}

\begin{table}[H]\footnotesize
  \[
    \begin{array}{r@{\;=\;}l|ll}
      \multicolumn{2}{c|}{\xi}
      & \multicolumn{1}{c}{p}
      & \multicolumn{1}{c}{\frac{\phi(p+1)}{\phi(p)}} \\ \midrule
      \frac \pi {10}
      & 0.314159265\dotso
      & \seqsplit{1902037158772097}
      & \underline{0.3141592}33\dots \\
      \frac \pi {20}
      & 0.1570796326794\dots
      & \seqsplit{65230510948153143551387418989}
      & \underline{0.15707961697}22\dots \\
      \frac e {10}
      & 0.2718281828\dots
      & \seqsplit{9240530296299581}
      & \underline{0.2718281}556\dots \\
      \frac 1 e
      & 0.367879441\dots
      & \seqsplit{5309646891817189}
      & \underline{0.3678794}04\dots \\
      \frac {\sqrt 2} {10}
      & 0.141421356\dots
      & \seqsplit{4002770936541226705231153047269}
      & \underline{0.1414213}42\dots \\
      \sqrt 2 - 1
      & 0.414213562\dots
      & \seqsplit{233570456771714761}
      & \underline{0.4142135}20\dots \\
      \frac {\sqrt 3} {10}
      & 0.173205080\dots
      & \seqsplit{721221963089661856995482309}
      & \underline{0.1732050}63\dots \\
      \frac {\sqrt 5} {10}
      & 0.223606797\dots
      & \seqsplit{1061017350953476949129}
      & \underline{0.2236067}75\dots \\
      \frac {\sqrt 7} {10}
      & 0.264575131\dots
      & \seqsplit{184295506315169}
      & \underline{0.2645751}04\dots \\
      \frac{\sqrt 5+1}{20}
      & 0.161803398\dots
      & \seqsplit{94721096130489558305686109}
      & \underline{0.1618033}82\dots \\
    \end{array}
  \]
  \caption{Numerical examples for Theorem~\ref{Theorem:SchinzelPrime}}
  \label{Table:SchinzelPrime}
\end{table}

\begin{table}[H]\footnotesize
  \[
    \begin{array}{r@{\;=\;}l|ll}
      \multicolumn{2}{c|}{\xi}
      & \multicolumn{1}{c}{p}
      & \multicolumn{1}{c}{\frac{\phi(p+1)}{\phi(p)}} \\ \midrule
      \frac \pi {10}
      & 0.314159265\dots
      & 1902037158772097
      & \underline{0.3141592}33\dots \\
      \pi-3
      & 0.141592653\dots
      & 37850921999916257860282849163969
      & \underline{0.1415926}39\dots \\
      \frac \pi {20}
      & 0.15707963\dots
      & 2762774807373943331969
      & \underline{0.157079}47\dots \\
      \frac e {10}
      & 0.2718281828\dots
      & 9240621837106421
      & \underline{0.2718281}556\dots \\
      \frac {\sqrt 2} {10}
      & 0.14142135623\dots
      & 4003599638847875898651948718589
      & \underline{0.141421342}09\dots \\
      \frac {\sqrt 3}{10}
      & 0.173205080\dots
      & 721231627248456517343440289
      & \underline{0.1732050}63\dots \\
      \frac {\sqrt 5}{10}
      & 0.223606797\dots
      & 1061064248215841845709
      & \underline{0.2236067}75\dots \\
      \frac {\sqrt 7}{10}
      & 0.264575131\dots
      & 184327276293689
      & \underline{0.2645751}04\dots \\
      \frac {\sqrt 5+1} {20}
      & 0.161803398\dots
      & 94721096130489558305686109
      & \underline{0.1618033}82\dots \\
    \end{array}
  \]
  \caption{Numerical examples for Theorem~\ref{Theorem:Schinzel3}}
  \label{Table:Schinzel3}
\end{table}

\begin{table}[H]\footnotesize
  \[
    \begin{array}{r@{\;=\;}l|m{\pwidth}l}
      \multicolumn{2}{c|}{\xi}
      & \multicolumn{1}{c}{p}
      & \multicolumn{1}{c}{\frac{\sigma(p+1)}{\sigma(p-1)}} \\ \midrule
      \gamma
      & 0.577215664901532\dots
      & \seqsplit{285659715165349572896202438481583805634555808139263750257162232532026578053941}
      & \underline{0.5772156649015}27\dots \\
      \pi
      & 3.1415\dots
      & \seqsplit{26469212222568419838609212640874686000298158238701096717482341268845579849569676131609566546557369367349279138922992270763544429}
      & \underline{3.14}07 \\
      \frac \pi 2
      & 1.57079632679489661\dots
      & \seqsplit{271445221589612296917454398280249786108341248083170868563698778989}
      & \underline{1.570796326794896}75\dots \\
      e
      & 2.718281\dots
      & \seqsplit{201426586111496390059982859052378729528572579361216852839150744818500432120229}
      & \underline{2.7182}77\dots \\
      \sqrt 2
      & 1.414213562373095\dots
      & \seqsplit{1300421337603424629894146725516266587797672109490592798687340114429}
      & \underline{1.4142135623730}89\dots \\
      \sqrt 3
      & 1.73205080756887729\dots
      & \seqsplit{5001304679832232346811636913292219681806190173052913912408611416987702190648767132709}
      & \underline{1.732050807568877}44\dots \\
      \sqrt 5
      & 2.2360679774997\dots
      & \seqsplit{20138170271534222884581768517669808849964402993328707829576024653820857989}
      & \underline{2.23606797749}80\dots \\
      \frac {\sqrt 5 + 1} {2}
      & 1.618033988749894\dots
      & \seqsplit{2768409745128994233528433338776786616521251020681096425364422069}
      & \underline{1.6180339887498}88\dots \\
      \frac {\sqrt 5 - 1} {2}
      & 0.6180339887498948\dots
      & \seqsplit{10854285409284254226574808906161529175312651234309760660428521447231541940154141}
      & \underline{0.61803398874989}54\dots \\
      \log 2
      & 0.6931471805599\dots
      & \seqsplit{158315132994548322759751536776201939944008616235544504225821}
      & \underline{0.6931471805}601\dots \\
      \log 3
      & 1.09861228866810969\dots
      & \seqsplit{212204529605325097826520886692665882479828302131195312459828945357}
      & \underline{1.098612288668109}89\dots
    \end{array}
  \]
  \caption{Numerical examples for
    Theorem~\ref{Theorem:Sigma}\ref{Part:SigmaPrimitive}.}
  \label{Table:SigmaPrimitive}
\end{table}

\begin{table}[H]\footnotesize
  \[
    \begin{array}{r@{\;=\;}l|m{\pwidth}l}
      \multicolumn{2}{c|}{\xi}
      & \multicolumn{1}{c}{p}
      & \multicolumn{1}{c}{\frac{\sigma(p+1)}{\sigma(p)}} \\ \midrule
      10\gamma
      & 5.77215\dots
      & \seqsplit{82878345908080079551179656289238279189797435695514974872920897457004916122789562389}
      & \underline{5.772}21\dots \\
      \pi
      & 3.1415926\dots
      & \seqsplit{203351964077675489}
      & \underline{3.14159}57\dots \\
      2 \pi
      & 6.283185\dots
      & \seqsplit{655193122132808448338702066088021133128673476003092904430789475344820749138959120591709442615982404569457060616806991988883790296018540849009}
      & \underline{6.2831}91\dots \\
      \frac \pi 2
      & 1.57079632\dots
      & \seqsplit{22955076440560177}
      & \underline{1.570796}48\dots \\
      e
      & 2.7182818\dots
      & \seqsplit{1716574977543884369}
      & \underline{2.71828}21\dots \\
      \sqrt 3
      & 1.73205080\dots
      & \seqsplit{156513047792653}
      & \underline{1.732050}98\dots \\
      \sqrt 5
      & 2.236067\dots
      & \seqsplit{4852141797161}
      & \underline{2.2360}70\dots \\
      \frac{\sqrt 5 + 1}{2}
      & 1.6180339\dots
      & \seqsplit{18106083326748793}
      & \underline{1.61803}41\dots \\
    \end{array}
  \]
  \caption{Numerical examples for
    Theorem~\ref{Theorem:Sigma}\ref{Part:SigmaSchinzelPrime}.}
  \label{Table:SigmaSchinzelPrime}
\end{table}

\begin{table}[H]\footnotesize
  \[
    \begin{array}{r@{\;=\;}l|m{\pwidth}l}
      \multicolumn{2}{c|}{\xi}
      & \multicolumn{1}{c}{p}
      & \multicolumn{1}{c}{\frac{\sigma(p+1)}{\sigma(p)}} \\ \midrule
      10\gamma
      & 5.7721566\dots
      & \seqsplit{2595684868506848043313701558218369992865598970206944461532321445567887878042009056424590269}
      & \underline{5.77215}72\dots \\
      \pi
      & 3.14159265\dots
      & \seqsplit{2007224256303311429}
      & \underline{3.141592}96\dots \\
      2\pi
      & 6.28318530\dots
      & \seqsplit{6561189247647575857894512417012071976627284592861948711670209688265586319709280421920002799868898085933789604550135691104643715461617425356689}
      & \underline{6.283185}93\dots \\
      e
      & 2.7182818\dots
      & \seqsplit{1717018302510268229}
      & \underline{2.71828}21\dots \\
      \sqrt 5
      & 2.2360679\dots
      & \seqsplit{285009842420045757101}
      & \underline{2.23606}82\dots \\
    \end{array}
  \]
  \caption{Numerical examples for
    Theorem~\ref{Theorem:Sigma}\ref{Part:SigmaSchinzel3}.}
  \label{Table:SigmaSchinzel3}
\end{table}

\bibliographystyle{plain}

\bibliography{PRBTP2}

\begin{thebibliography}{10}

\bibitem{1CRTA}
S.L. Aletheia-Zomlefer, L.~Fukshansky, and S.R. Garcia.
\newblock The {B}ateman--{H}orn conjecture: Heuristics, history, and
  applications.
\newblock {\em Expo. Math.}
\newblock (in press) \url{https://arxiv.org/abs/1807.08899}.

\bibitem{Bateman}
P.T. Bateman and R.A. Horn.
\newblock A heuristic asymptotic formula concerning the distribution of prime
  numbers.
\newblock {\em Math. Comp.}, 16:363--367, 1962.

\bibitem{Bateman2}
P.T. Bateman and R.A. Horn.
\newblock Primes represented by irreducible polynomials in one variable.
\newblock In {\em Proc. {S}ympos. {P}ure {M}ath., {V}ol. {VIII}}, pages
  119--132. Amer. Math. Soc., Providence, R.I., 1965.

\bibitem{Chen2}
J.-R. Chen.
\newblock On the representation of a larger even integer as the sum of a prime
  and the product of at most two primes.
\newblock {\em Sci. Sinica}, 16:157--176, 1973.

\bibitem{DeKoninck}
J.-M. De~Koninck and F.~Luca.
\newblock {\em Analytic number theory}, volume 134 of {\em Graduate Studies in
  Mathematics}.
\newblock American Mathematical Society, Providence, RI, 2012.
\newblock Exploring the anatomy of integers.

\bibitem{Dickson}
L.E. Dickson.
\newblock A new extension of {D}irichlet's theorem on prime numbers.
\newblock {\em Messenger of mathematics}, 33:155--161, 1904.

\bibitem{Erdos}
P.~Erd\H{o}s.
\newblock Some remarks on {E}uler's {$\varphi $} function.
\newblock {\em Acta Arith.}, 4:10--19, 1958.

\bibitem{Erdos2}
P.~Erd\H{o}s, K.~Gy\H{o}ry, and Z.~Papp.
\newblock On some new properties of functions {$\sigma (n)$}, {$\varphi (n)$},
  {$d(n)$} and {$\nu (n)$}.
\newblock {\em Mat. Lapok}, 28(1-3):125--131, 1980.

\bibitem{Friedlander}
J.~Friedlander and H.~Iwaniec.
\newblock {\em Opera de cribro}, volume~57 of {\em American Mathematical
  Society Colloquium Publications}.
\newblock American Mathematical Society, Providence, RI, 2010.

\bibitem{PRBTP}
S.R. Garcia, E.~Kahoro, and F.~Luca.
\newblock Primitive {R}oot {B}ias for {T}win {P}rimes.
\newblock {\em Exp. Math.}, 28(2):151--160, 2019.

\bibitem{GL}
S.R. Garcia and F.~Luca.
\newblock On the difference in values of the {E}uler totient function near
  prime arguments.
\newblock In {\em Irregularities in the distribution of prime numbers}, pages
  69--96. Springer, Cham, 2018.

\bibitem{GLS}
S.R. Garcia, F.~Luca, and T.~Schaaff.
\newblock Primitive root biases for prime pairs {I}: {E}xistence and
  non-totality of biases.
\newblock {\em J. Number Theory}, 185:93--120, 2018.

\bibitem{HL}
G.H. Hardy and J.E. Littlewood.
\newblock Some problems of `{P}artitio numerorum'; {III}: {O}n the expression
  of a number as a sum of primes.
\newblock {\em Acta Math.}, 114(3):215--273, 1923.

\bibitem{Hardy}
G.H. Hardy and E.M. Wright.
\newblock {\em An introduction to the theory of numbers}.
\newblock Oxford University Press, Oxford, sixth edition, 2008.
\newblock Revised by D. R. Heath-Brown and J. H. Silverman, With a foreword by
  Andrew Wiles.

\bibitem{Chen1}
Chen J.-R.
\newblock On the representation of a large even integer as the sum of a prime
  and the product of at most two primes.
\newblock {\em Kexue Tongbao (Foreign Lang. Ed.)}, 17:385--386, 1966.

\bibitem{Maynard}
J.~Maynard.
\newblock Small gaps between primes.
\newblock {\em Ann. of Math. (2)}, 181(1):383--413, 2015.

\bibitem{Mertens}
F.~Mertens.
\newblock Ein {B}eitrag zur analytischen {Z}ahlentheorie.
\newblock {\em J. Reine Angew. Math.}, 78:46--62, 1874.

\bibitem{Polymath}
D.H.J. Polymath.
\newblock New equidistribution estimates of {Z}hang type.
\newblock {\em Algebra Number Theory}, 8(9):2067--2199, 2014.

\bibitem{Polymath2}
D.H.J. Polymath.
\newblock Variants of the {S}elberg sieve, and bounded intervals containing
  many primes.
\newblock {\em Res. Math. Sci.}, 1:Art. 12, 83, 2014.

\bibitem{Ribenboim}
P.~Ribenboim.
\newblock {\em The new book of prime number records}.
\newblock Springer-Verlag, New York, 1996.

\bibitem{Sandor}
J.~S\'{a}ndor, D.S. Mitrinovi\'{c}, and B.~Crstici.
\newblock {\em Handbook of number theory. {I}}.
\newblock Springer, Dordrecht, 2006.
\newblock Second printing of the 1996 original.

\bibitem{Schinzel}
A.~Schinzel.
\newblock Generalisation of a theorem of {B}.{S}.{K}.{R}. {S}omayajulu on the
  {E}uler's function {$\phi(n)$}.
\newblock {\em Ganita}, 5:123--128 (1955), 1954.

\bibitem{Schinzel2}
A.~Schinzel.
\newblock Quelques th\'{e}or\`emes sur les fonctions {$\varphi(n)$} et
  {$\sigma(n)$}.
\newblock {\em Bull. Acad. Polon. Sci. Cl. III.}, 2:467--469 (1955), 1954.

\bibitem{Schinzel3}
A.~Schinzel.
\newblock On functions {$\varphi(n)$} and {$\sigma(n)$}.
\newblock {\em Bull. Acad. Polon. Sci. Cl. III.}, 3:415--419, 1955.

\bibitem{SS}
A.~Schinzel and W.~Sierpi\'{n}ski.
\newblock Sur quelques propri\'{e}t\'{e}s des fonctions {$\varphi(n)$} et
  {$\sigma(n)$}.
\newblock {\em Bull. Acad. Polon. Sci. Cl. III.}, 2:463--466 (1955), 1954.

\bibitem{SW}
A.~Schinzel and Y.~Wang.
\newblock A note on some properties of the functions {$\varphi (n)$}, {$\sigma
  (n)$} and {$\theta (n)$}.
\newblock {\em Ann. Polon. Math.}, 4:201--213, 1958.

\bibitem{Zhang}
Y.~Zhang.
\newblock Bounded gaps between primes.
\newblock {\em Ann. of Math. (2)}, 179(3):1121--1174, 2014.

\end{thebibliography}

\end{document}